\documentclass[11pt]{article}
\usepackage[margin=1in]{geometry}
\usepackage{amsmath,amssymb,amsthm}
\usepackage[dvipsnames]{xcolor}
\usepackage{hyperref}
\hypersetup{pdftitle={A quadratic upper bound on the Chv\'atal rank of polytopes in the 0/1-cube},pdfauthor={Alberto Del Pia},colorlinks=true,linkcolor=blue!50!black,citecolor=blue!50!black,urlcolor=blue!50!black}

\newtheorem{theorem}{Theorem}
\newtheorem{lemma}[theorem]{Lemma}
\newtheorem{corollary}[theorem]{Corollary}
\newtheorem{claim}{Claim}

\newenvironment{cpf}
{\begin{trivlist} \item[] {\em Proof of claim. }}
{$\hfill\diamond$ \end{trivlist}}

\theoremstyle{definition}
\newtheorem{fact}{Fact}

\newcommand{\Z}{\mathbb{Z}}
\newcommand{\R}{\mathbb{R}}
\newcommand{\PI}{P_I}
\newcommand{\norm}[1]{\lVert #1\rVert}
\newcommand{\cube}{[0,1]^n}
\newcommand{\rk}{\operatorname{rk}}
\newcommand{\dep}{\operatorname{depth}}
\DeclareMathOperator{\relint}{relint}

\title{A quadratic upper bound on the Chv\'atal rank\\ of polytopes in the $0/1$-cube}
\author{Alberto Del Pia\thanks{Department of Industrial and Systems Engineering \& Wisconsin Institute for Discovery, University of Wisconsin--Madison, Madison, WI, USA. E-mail: \texttt{delpia@wisc.edu}}}
\date{September 25, 2026}

\begin{document}
\maketitle

\begin{abstract}
\noindent We show that every polytope $P\subseteq[0,1]^n$, and more generally every compact convex set, has Chv\'atal rank at most $12.22\,n^2+n\log_2 n+2n+4$. This improves the $O(n^2\log n)$ bound of Eisenbrand and Schulz and, together with the $\Omega(n^2)$ lower bound of Rothvo\ss{} and Sanit\`a, shows that the maximum Chv\'atal rank of a polytope in $[0,1]^n$ is $\Theta(n^2)$. More precisely, if $P$ contains an integer point, then for every $c\in\Z^n\setminus\{0\}$ the inequality $cx\le\max\{cy: y\in P\cap\Z^n\}$ is valid for the $k$-th Chv\'atal closure of $P$ for some $k\le 12.22\,n^2+2n+2\log_2\norm c_\infty+4$. Following Eisenbrand and Schulz, we derive this inequality along a chain of coarser and coarser integer vectors, but instead of halving the vector in each step, we round $\tau c$ for a scale $\tau$ chosen freely in each dyadic window $[2^{-t-1},2^{-t}]$. The main new ingredient is a multiscale version of Dirichlet's approximation theorem, proved by an elementary volume argument: for every $c\in\R^n$, the $\ell_1$-distances of $\tau c$ to $\Z^n$, minimized within each dyadic window and summed over all windows, total less than $1.222\,n^2$, independently of $\norm c_\infty$.


\end{abstract}

\section{Introduction}

Chv\'atal--Gomory cuts are a basic tool of integer programming, and the Chv\'atal rank of a polytope measures how many rounds of them are needed to reach its integer hull. Chv\'atal~\cite{Chv73} showed that the rank of every rational polytope is finite, and Schrijver~\cite{Sch80} extended this to rational polyhedra. In general the rank can be arbitrarily large already in dimension two~\cite[p.~344]{Sch86}. For polytopes in the cube $[0,1]^n$, however, it is bounded in terms of $n$ alone~\cite{BocEisHarSch99}. The previous best upper bound, $O(n^2\log n)$, is due to Eisenbrand and Schulz~\cite{EisSch03}, and the best lower bound, $\Omega(n^2)$, is due to Rothvo\ss{} and Sanit\`a~\cite{RotSan13,RotSan17}. We close this gap.

For a compact convex set $P\subseteq\R^n$ its \emph{Chv\'atal closure} is
\[
P':=\bigcap_{a\in\Z^n}\bigl\{x\in\R^n: ax\le\lfloor\max_{y\in P}ay\rfloor\bigr\},
\]
where $ax:=\sum_{\ell=1}^na_\ell x_\ell$ denotes the standard inner product of $a,x\in\R^n$. We set $P':=\emptyset$ if $P=\emptyset$. Let $\PI:=\operatorname{conv}(P\cap\Z^n)$ be the \emph{integer hull} of $P$, which is a polytope since $P\cap\Z^n$ is finite. It is well known that $\PI\subseteq P'\subseteq P$. In particular, $P'$ is again a compact convex set, and we can iterate: let $P^{(0)}:=P$ and $P^{(k+1)}:=(P^{(k)})'$ for every integer $k\ge0$. We call $P^{(k)}$ the \emph{$k$-th Chv\'atal closure} of $P$. The \emph{Chv\'atal rank} $\rk(P)$ is the least $k$ with $P^{(k)}=\PI$.

If $\PI\ne\emptyset$, then for $v\in\Z^n$ we write $h(v):=\max_{x\in\PI}vx$, which is an integer, and $\dep_P(v)$ denotes the least $k$ such that $vx\le h(v)$ is valid for $P^{(k)}$. Following~\cite{EisSch03}, we say that $v$ is saturated with respect to $P^{(k)}$ if $vx\le h(v)$ is valid for $P^{(k)}$.

The Chv\'atal rank is usually studied for polytopes, but our main result holds more generally for compact convex sets.

\begin{theorem}\label{thm:main}
Let $P\subseteq\cube$ be a compact convex set with $\PI\ne\emptyset$, and let $c\in\Z^n\setminus\{0\}$, $N:=\norm c_\infty$. Then
\[
\dep_P(c)\ \le\ 12.22\,n^2+2n+2\log_2N+4.
\]
Moreover, for every compact convex set $P\subseteq\cube$,
\[
\rk(P)\ \le\ 12.22\,n^2+n\log_2n+2n+4.
\]
\end{theorem}

Together with the lower bound of~\cite{RotSan17}, Theorem~\ref{thm:main} shows that the maximum Chv\'atal rank of a polytope in $[0,1]^n$ is $\Theta(n^2)$. We have not tried to optimize the constants, since our focus is the order of growth.

\paragraph{Previous work.} Chv\'atal, Cook and Hartmann~\cite{ChvCooHar89} gave polytopes in the cube with empty integer hull and rank at least $n$, which is exactly $n$ by~\cite{BocEisHarSch99}. They also gave lower bounds, none exceeding $n$, for relaxations of natural combinatorial problems. Bockmayr, Eisenbrand, Hartmann and Schulz~\cite{BocEisHarSch99} proved the first polynomial upper bound, $\rk(P)=O(n^3\log n)$, via $\dep_P(c)\le n^2(1+\lfloor\log_2\norm c_\infty\rfloor)$ for $c\in\Z^n\setminus\{0\}$. Eisenbrand and Schulz~\cite{EisSch03} improved both bounds to
\[
\dep_P(c)\le n^2+2n\bigl(1+\lfloor\log_2\norm c_\infty\rfloor\bigr),\qquad \rk(P)\le n^2\bigl(2+\lfloor\log_2 n\rfloor\bigr),
\]
and also proved $\dep_P(c)\le n+\norm c_1$ and a lower bound $(1+\varepsilon)n$ for infinitely many $n$. Pokutta and Stauffer~\cite{PokSta11} raised the lower bound to $(1+1/e-o(1))n$. Pokutta and Schulz~\cite{PokSch11} characterized the integer-empty polytopes of rank $n$. Rothvo\ss{} and Sanit\`a~\cite{RotSan13,RotSan17} constructed polytopes of rank $\Omega(n^2)$ (see also the textbook~\cite[Section~5.2.2]{ConCorZam14}).

The $\log n$ in the upper bound comes only from the size of facet coefficients: the facet normals of $\PI$ can require coefficients with $\Theta(n\log n)$ bits~\cite{AloVu97,Zie00}, and \cite{EisSch03} spends $\Theta(n)$ rounds per bit. Hence the $\log n$ cannot be removed by a better upper bound on the coefficients, and a better depth bound is needed. Cornu\'ejols and Lee~\cite{CorLee16}, \cite[Section~7]{CorLee18} obtained $O(n^2\log\log n)$ for polytopes that miss at most quasi-polynomially many $0/1$ points, by showing that then $\log_2\norm c_\infty=O(n\log\log n)$ for the facet normals. The coefficients in the construction of~\cite{RotSan17} are of order $2^{\Theta(n)}$. In the concluding remarks of~\cite{RotSan17}, Rothvo\ss{} and Sanit\`a observe that beyond this range Dirichlet's theorem on simultaneous Diophantine approximation is an obstacle to extending their lower bound, and that ``this insight could potentially be used to improve the upper bound on the Chv\'atal rank''. Our proof carries out this idea, as explained in the overview below.

\paragraph{Overview of the proof.} As in \cite{EisSch03}, we saturate $c$ through a chain of coarser and coarser integer vectors $c=v_0,v_1,\dots$ ending at $0$, together with positive multipliers $\mu_t$ such that $v_t$ is close in $\ell_1$-norm to $\mu_tv_{t+1}$. If $v_{t+1}$ is already saturated with respect to $P^{(k)}$, then by Lemma~\ref{lem:jump} the inequality $v_tx\le h(v_t)+\norm{v_t-\mu_tv_{t+1}}_1$ is valid for $P^{(k)}$. Each unit of the excess $\norm{v_t-\mu_tv_{t+1}}_1$ costs two further applications of the closure, and the induction over the faces of the cube adds only $2n$ in total. This is Lemma~\ref{lem:pipeline}: $\dep_P(c)\le2n+2\sum_tE_t$, where $E_t$ is the error of the $t$-th step rounded up. Lemmas~\ref{lem:jump} and~\ref{lem:pipeline} are proved in Section~\ref{sec:chains}.

It remains to choose the chain so that its total error is small. Eisenbrand and Schulz take $v_{t+1}=\lfloor v_t/2\rfloor$ and $\mu_t=2$, so each of the roughly $\log_2\norm c_\infty$ steps can have error close to $n$. The errors then add up to about $n\log_2\norm c_\infty$, which is of order $n^2\log n$ for the facet normals of $\PI$ in the worst case. Instead, we take $v_t$ to be a rounding of $\tau_tc$, where $\tau_t$ is chosen freely in the dyadic window $[2^{-t-1},2^{-t}]$ so that $\tau_tc$ is as close as possible to $\Z^n$. The main new ingredient is Theorem~\ref{thm:dirichlet}: for every $c\in\R^n$ and $0<\rho\le n$, at most $n\log_2(n/\rho)$ windows contain no $\tau$ for which the $\ell_1$-distance of $\tau c$ to $\Z^n$ is at most $\rho$. Consequently the smallest such distances, one for each window, sum to $O(n^2)$ independently of $\norm c_\infty$, which is Corollary~\ref{cor:sum}, and hence so does the total error along the chain. The proof of Theorem~\ref{thm:dirichlet} is an elementary volume argument that treats all windows at once. Theorem~\ref{thm:dirichlet} and Corollary~\ref{cor:sum} are proved in Section~\ref{sec:dirichlet}, which concerns only a vector $c\in\R^n$ and involves no polytope.

In Section~\ref{sec:proof} we combine these results to prove Theorem~\ref{thm:main}. The bound on the rank follows since the facet normals of $\PI$ have coefficients of absolute value at most $n^{n/2}$. We close this section with some known facts that we will use.

\begin{fact}[{\cite[Theorem~1]{Sch80}}]\label{f:rational} If $P$ is a rational polytope, so is $P'$.\end{fact}
\begin{fact}[{\cite[p.~340]{Sch86}}]\label{f:face} If $F$ is a face of a rational polytope $P$, then $F'=P'\cap F$.\end{fact}
\begin{fact}[{\cite[Lemma~3]{BocEisHarSch99}}]\label{f:empty} If $P\subseteq\cube$ is a rational polytope with $\PI=\emptyset$, then $P^{(n)}=\emptyset$.\end{fact}
\begin{fact}[{\cite{PadGro85}}]\label{f:hadamard}
Let $P\subseteq\cube$ be a polytope with $\PI\ne\emptyset$. There are $a_1,\dots,a_m\in\Z^n\setminus\{0\}$ and $b_1,\dots,b_m\in\Z$ such that $\PI=\{x\in\R^n: a_ix\le b_i,\ i=1,\dots,m\}$ and $\norm{a_i}_\infty\le n^{n/2}$ for all $i$.
\end{fact}

\section{Chains of integer vectors}\label{sec:chains}

A \emph{chain} is a sequence $v_0,v_1,\dots,v_T\in\Z^n$ of integer vectors together with positive multipliers $\mu_0,\dots,\mu_{T-1}$. Our first lemma shows that if $w$ is saturated and $v$ is close in $\ell_1$-norm to a positive multiple of $w$, then $v$ is almost saturated.

\begin{lemma}\label{lem:jump}
Let $P\subseteq\cube$ be a polytope with $\PI\neq\emptyset$ and $h$ as above. Let $Q\subseteq\cube$ be any set, $v,w\in\Z^n$ and $\mu>0$, and suppose $wx\le h(w)$ for all $x\in Q$. Then $vx\le h(v)+\norm{v-\mu w}_1$ for all $x\in Q$.
\end{lemma}
\begin{proof}
Put $r:=v-\mu w$. For every $y\in\PI$ we have $\mu wy=vy-ry\le h(v)-ry$. Taking the maximum over $y\in\PI$, and using $\mu>0$, gives
\[
\mu h(w)\le h(v)+\max_{y\in\PI}-ry .
\]
Hence, for every $x\in Q$,
\[
vx=\mu wx+rx\le\mu h(w)+rx\le h(v)+\max_{y\in\PI}r(x-y)\le h(v)+\norm r_1,
\]
where the last inequality holds because $x,y\in\cube$ implies $|x_\ell-y_\ell|\le1$ for all $\ell$.
\end{proof}

Lemma~\ref{lem:jump} generalizes the bound on the integrality gap in the proof of \cite[Proposition~3.2]{EisSch03}, which is the case $v\ge0$, $w=\lfloor v/2\rfloor$, $\mu=2$.

The next lemma, the main result of this section, bounds the depth of an integer vector $c$ in terms of the $\ell_1$-errors $\norm{v_t-\mu_tv_{t+1}}_1$ along a chain starting at $c$.

\begin{lemma}\label{lem:pipeline}
Let $P\subseteq\cube$ be a rational polytope with $\PI\neq\emptyset$ and $h$ as above. Let $c\in\Z^n$, and let $v_0,\dots,v_T$ with multipliers $\mu_0,\dots,\mu_{T-1}$ be a chain with $v_0=c$. Put $E_t:=\lceil\norm{v_t-\mu_tv_{t+1}}_1\rceil$ for $t<T$, and $E_T:=\lceil\max_{y\in P}v_Ty-h(v_T)\rceil\ge0$. Then
\[
\dep_P(c)\ \le\ 2n+2\sum_{t=0}^{T}E_t .
\]
\end{lemma}
\begin{proof}
The idea is to lower the right-hand side of $v_tx\le h(v_t)+E_t$ one unit at a time on the faces of $\cube$, for $t=T,T-1,\dots,0$, passing from $v_{t+1}$ to $v_t$ by Lemma~\ref{lem:jump}. Claims~\ref{cl:zero}--\ref{cl:step} give recursive bounds on the number of rounds this requires, and Claim~\ref{cl:unroll} together with an induction combines them.

By Fact~\ref{f:rational}, every $P^{(k)}$ is a rational polytope. A \emph{$d$-face} of $\cube$ is a face of dimension $d$. For integers $0\le t\le T$, $0\le d\le n$ and $0\le j\le E_t$ let $S(t,d,j)$ be the statement
\emph{``$v_tx\le h(v_t)+E_t-j$ holds on $P^{(k)}\cap G$ for every $d$-face $G$ of $\cube$''},
and let $f(t,d,j)\in\{0,1,2,\dots\}\cup\{\infty\}$ be the least $k$ for which it holds. Since $P^{(k+1)}\subseteq P^{(k)}$, it then holds for all larger $k$.
\begin{claim}\label{cl:zero}
$f(t,0,j)=0$.
\end{claim}
\begin{cpf}
A $0$-face of $\cube$ is a set $\{x\}$ with $x\in\{0,1\}^n$; if $x\in P$, then $x\in\PI$, so $v_tx\le h(v_t)\le h(v_t)+E_t-j$.
\end{cpf}

\begin{claim}\label{cl:jump}
$f(t,d,0)\le f(t+1,d,E_{t+1})$ for $t<T$, and $f(T,d,0)=0$.
\end{claim}
\begin{cpf}
Let $t<T$ and assume that $k:=f(t+1,d,E_{t+1})$ is finite. Let $G$ be a $d$-face of $\cube$. By $S(t+1,d,E_{t+1})$, $v_{t+1}x\le h(v_{t+1})$ for all $x\in Q:=P^{(k)}\cap G$, so Lemma~\ref{lem:jump} with $v=v_t$, $w=v_{t+1}$, $\mu=\mu_t$ gives $v_tx\le h(v_t)+\norm{v_t-\mu_tv_{t+1}}_1\le h(v_t)+E_t$ on $Q$, which is $S(t,d,0)$. Moreover $f(T,d,0)=0$, since for every $x\in P$ we have $v_Tx\le\max_{y\in P}v_Ty\le h(v_T)+E_T$ by definition of $E_T$.
\end{cpf}

\begin{claim}\label{cl:step}
For $d\ge1$ and $0\le j<E_t$: $f(t,d,j+1)\le\max\{f(t,d,j),f(t,d-1,j+1)\}+2$.
\end{claim}
\begin{cpf}
Assume that $k:=\max\{f(t,d,j),f(t,d-1,j+1)\}$ is finite. Let $G$ be a $d$-face of $\cube$, and let $\gamma:=h(v_t)+E_t-j\in\Z$. By $S(t,d,j)$, since $k\ge f(t,d,j)$, we have $v_tx\le\gamma$ on $P^{(k)}\cap G$. Let $F:=P^{(k)}\cap G\cap\{x\in\R^n: v_tx=\gamma\}$. If $F\ne\emptyset$, it is a face of $P^{(k)}\cap G$, and so of $P^{(k)}$.

We have $F\subseteq\relint G$. Indeed, as $d\ge1$, the relative boundary of $G$ is the union of its facets, which are $(d-1)$-faces of $\cube$. For each such facet $G'$, $v_tx\le\gamma-1$ holds on $P^{(k)}\cap G'$ by $S(t,d-1,j+1)$, since $k\ge f(t,d-1,j+1)$. As $v_tx=\gamma$ on $F$, no point of $F$ lies on the relative boundary of $G$.

We show that $P^{(k+1)}\cap F=\emptyset$. This is clear if $F=\emptyset$. Otherwise $F$ is a face of $P^{(k)}$, so Fact~\ref{f:face} gives $P^{(k+1)}\cap F=F'$, and it suffices to show $F'=\emptyset$. Pick a coordinate $x_\ell$ that is not fixed on $G$, which exists since $d\ge1$. As $F$ is a nonempty rational polytope contained in $\relint G$, we have $0<\min_{x\in F}x_\ell\le\max_{x\in F}x_\ell<1$, so $x_\ell\le0$ and $x_\ell\ge1$ are valid for $F'$, and $F'=\emptyset$.

Thus $v_tx<\gamma$ on $P^{(k+1)}\cap G$. If $P^{(k+1)}\cap G=\emptyset$, then $P^{(k+2)}\cap G=\emptyset$ as well. Otherwise the maximum of $v_tx$ over the polytope $P^{(k+1)}\cap G$ is attained, hence smaller than $\gamma$, and since $P^{(k+1)}\cap G$ is a face of $P^{(k+1)}$, Fact~\ref{f:face} and the integrality of $\gamma$ give $v_tx\le\gamma-1$ on $(P^{(k+1)}\cap G)'=P^{(k+2)}\cap G$. In both cases $v_tx\le\gamma-1$ on $P^{(k+2)}\cap G$, and since $\gamma-1=h(v_t)+E_t-(j+1)$, this is $S(t,d,j+1)$ with $k+2$ in place of $k$.
\end{cpf}

To combine Claims~\ref{cl:zero}--\ref{cl:step}, we list all pairs $(t,j)$ with $0\le t\le T$ and $0\le j\le E_t$ in the order $(T,0),(T,1),\dots,(T,E_T),(T-1,0),\dots,(0,E_0)$, and number them $i=1,\dots,M$, where $\Sigma:=\sum_{t=0}^TE_t$ and $M:=\Sigma+T+1$. Let $g(i,d):=f(t,d,j)$ for the $i$-th pair $(t,j)$, let $g(0,d):=0$, and let $\sigma(i)$ be the number of pairs with $j\ge1$ among the first $i$.

\begin{claim}\label{cl:unroll}
Let $1\le i\le M$ and let $(t,j)$ be the $i$-th pair. Then $g(i,0)=0$. If $j=0$, then $g(i,d)\le g(i-1,d)$ for $0\le d\le n$. If $j\ge1$, then $g(i,d)\le\max\{g(i-1,d),g(i,d-1)\}+2$ for $1\le d\le n$.
\end{claim}
\begin{cpf}
The equality $g(i,0)=f(t,0,j)=0$ is Claim~\ref{cl:zero}. Let $j=0$. If $i=1$, then $(t,j)=(T,0)$ and $g(1,d)=f(T,d,0)=0$ by Claim~\ref{cl:jump}. If $i\ge2$, then $t<T$ and the $(i-1)$-st pair is $(t+1,E_{t+1})$, so Claim~\ref{cl:jump} gives $g(i,d)=f(t,d,0)\le f(t+1,d,E_{t+1})=g(i-1,d)$. If $j\ge1$, then $(t,j-1)$ is the $(i-1)$-st pair, and Claim~\ref{cl:step} with $j-1$ in place of $j$ gives the inequality.
\end{cpf}

Induction on $i+d$ gives $g(i,d)\le2(\sigma(i)+d)$ for $0\le i\le M$ and $0\le d\le n$. This holds if $i=0$, since $g(0,d)=0$ and $\sigma(0)=0$, and if $d=0$, since $g(i,0)=0$ by Claim~\ref{cl:unroll}. For $i,d\ge1$ it follows from Claim~\ref{cl:unroll}, since $\sigma(i)=\sigma(i-1)$ if the $i$-th pair has $j=0$, and $\sigma(i)=\sigma(i-1)+1$ otherwise.

The last pair is $(0,E_0)$ and $\sigma(M)=\Sigma$, so $f(0,n,E_0)=g(M,n)\le2(\Sigma+n)$. The only $n$-face of $\cube$ is $\cube\supseteq P^{(k)}$, so $S(0,n,E_0)$ says that $cx\le h(c)$ on $P^{(k)}$. Hence $\dep_P(c)\le2(\Sigma+n)$.
\end{proof}

\begin{corollary}\label{cor:ES}
Let $P\subseteq\cube$ be a rational polytope with $\PI\neq\emptyset$. Then $\dep_P(c)\le 2n(\lfloor\log_2\norm c_\infty\rfloor+2)$ for every $c\in\Z^n\setminus\{0\}$.
\end{corollary}
\begin{proof}
Let $v_{t+1}$ be the vector obtained from $v_t/2$ by rounding each coordinate toward $0$, and let $\mu_t:=2$. Then $v_t-2v_{t+1}\in\{-1,0,1\}^n$, so $E_t\le n$, and $v_T=0$ for $T=\lfloor\log_2\norm c_\infty\rfloor+1$, so $E_T=0$. Apply Lemma~\ref{lem:pipeline}.
\end{proof}

Lemma~\ref{lem:pipeline} refines the argument in the proof of \cite[Proposition~3.2]{EisSch03}, which, for $c\ge0$, uses the chain of vectors $\lfloor c/2^t\rfloor$, $t=0,1,\dots$, with all multipliers equal to~$2$. There are two differences. First, Lemma~\ref{lem:pipeline} applies to an arbitrary chain, and its bound depends on the actual errors $E_t$, while the argument in \cite{EisSch03} uses, at every step, the worst-case bound $n$ on the integrality gap that holds for the halving chain. This flexibility is what we exploit in Section~\ref{sec:proof}. Second, in \cite{EisSch03} the bound on $v_tx$ over $P^{(k)}$ starts to be lowered only after $v_tx\le h(v_t)$ holds on $P^{(k)}\cap G$ for every facet $G$ of $\cube$. This ensures the hypothesis of \cite[Lemma~3.1]{EisSch03}, and it is established by induction on the dimension. In Claim~\ref{cl:step}, instead, to lower the bound on a face $G$ of $\cube$ by one it suffices that the bound on the facets of $G$ has already been lowered by one. Thus the bounds on faces of all dimensions decrease simultaneously, and the induction over the dimension contributes only the additive term $2n$ in Lemma~\ref{lem:pipeline}, instead of $n^2$ in \cite[Proposition~3.2]{EisSch03}.

Corollary~\ref{cor:ES}, which uses essentially the halving chain, slightly sharpens the bound $n^2+2n(1+\lfloor\log_2\norm c_\infty\rfloor)$ of \cite{EisSch03} for $n\ge3$. Combined with Facts~\ref{f:empty} and~\ref{f:hadamard} and the reduction at the start of Section~\ref{sec:proof}, it gives $\rk(P)\le n^2\log_2n+4n$ for every compact convex set $P\subseteq\cube$, slightly better than the bound $n^2(2+\lfloor\log_2n\rfloor)$ of \cite{EisSch03} for $n\ge3$.

\section{A multiscale Dirichlet theorem}\label{sec:dirichlet}

For integers $j\ge0$ let $I_j:=[2^{-j-1},2^{-j}]$ be the \emph{dyadic windows}. In Section~\ref{sec:proof}, the vectors of our chain are roundings of $\tau_tc$ with $\tau_t\in I_t$. In this section we show that the scales $\tau_t$ can be chosen so that the $\ell_1$-distances of the vectors $\tau_tc$ to $\Z^n$ sum to $O(n^2)$, independently of $\norm c_\infty$.

For $c\in\R^n$ and $\tau\in\R$ let $D_c(\tau):=\min_{z\in\Z^n}\norm{\tau c-z}_1$ be the $\ell_1$-distance of $\tau c$ to $\Z^n$. Since every real number is within $\frac12$ of an integer, $D_c(\tau)\in[0,n/2]$. The function $D_c$ is continuous and even, i.e., $D_c(-\tau)=D_c(\tau)$.

The following theorem shows that, for every $c$, only few dyadic windows $I_j$ contain no $\tau$ for which $\tau c$ is close to $\Z^n$ in $\ell_1$-norm.

\begin{theorem}\label{thm:dirichlet}
Let $c\in\R^n$ and $0<\rho\le n$, and let $B_\rho:=\{j\ge0: D_c(\tau)>\rho\ \text{for all } \tau\in I_j\}$. Then
\[
|B_\rho|\ \le\ n\log_2(n/\rho).
\]
\end{theorem}
\begin{proof}
We write $\operatorname{vol}(Y)$ for the Lebesgue measure of a measurable set $Y\subseteq\R$ or $Y\subseteq\R^n$, and we let $\sigma:=\rho/n\in(0,1]$.

We first construct a set $X\subseteq[0,1)$ with $\operatorname{vol}(X)\ge\sigma^n$, and then show that it satisfies
\begin{equation}\label{eq:avoid}
|\tau-\tau'|\notin \textstyle\bigcup_{j\in B_\rho}I_j\qquad\text{for all } \tau,\tau'\in X.
\end{equation}
For $z\in\R^n$ let $X_z:=\{\tau\in[0,1): \tau c-z\in[0,\sigma)^n+\Z^n\}$. The set of pairs $(\tau,z)\in[0,1)\times[0,1)^n$ with $\tau c-z\in[0,\sigma)^n+\Z^n$ is the union over $k\in\Z^n$ of the convex sets of pairs with $\tau c-z-k\in[0,\sigma)^n$. Only finitely many of these are nonempty, since $\tau c-z$ ranges over a bounded set. Hence this set is measurable, and exchanging the order of integration (Fubini's theorem) gives $\int_{[0,1)^n}\operatorname{vol}(X_z)\,dz=\int_0^1\operatorname{vol}\{z\in[0,1)^n: \tau c-z\in[0,\sigma)^n+\Z^n\}\,d\tau$. For fixed $\tau$, the condition on $z$ splits into the conditions $z_\ell\in(\tau c_\ell-\sigma,\tau c_\ell]+\Z$, $\ell=1,\dots,n$. As $\sigma\le1$, each of them is satisfied by a subset of $[0,1)$ of length $\sigma$. So the integrand on the right-hand side equals $\sigma^n$ for every $\tau$, and $\int_{[0,1)^n}\operatorname{vol}(X_z)\,dz=\sigma^n$. Since $[0,1)^n$ has volume $1$, there is $z\in[0,1)^n$ such that $X:=X_z$ satisfies $\operatorname{vol}(X)\ge \sigma^n$.

We now show that $X$ satisfies~\eqref{eq:avoid}. Let $\tau,\tau'\in X$. Then $\tau c-z=q+k$ and $\tau'c-z=q'+k'$ with $q,q'\in[0,\sigma)^n$ and $k,k'\in\Z^n$. Subtracting, $(\tau-\tau')c-(k-k')=q-q'$, so $D_c(\tau-\tau')\le\norm{q-q'}_1<n\sigma=\rho$. Since $D_c$ is even, this gives $D_c(|\tau-\tau'|)<\rho$. So, by the definition of $B_\rho$, no window $I_j$ with $j\in B_\rho$ contains $|\tau-\tau'|$, which is~\eqref{eq:avoid}.

For $\lambda>0$ let $\Phi(\lambda):=\sup_{\xi\in\R}\operatorname{vol}(X\cap[\xi,\xi+\lambda])$. Since $X\subseteq[0,1)$, we have $\operatorname{vol}(X)\le\Phi(1)$. Clearly $\Phi(\lambda)\le\lambda$, and $\Phi(2\lambda)\le2\Phi(\lambda)$ since $[\xi,\xi+2\lambda]=[\xi,\xi+\lambda]\cup[\xi+\lambda,\xi+2\lambda]$. We now show that $\Phi(2^{-j})\le\Phi(2^{-j-1})$ for every $j\in B_\rho$. Let $j\in B_\rho$, $\xi\in\R$, and $Y:=X\cap[\xi,\xi+2^{-j}]$. Any two points of $Y$ are at distance at most $2^{-j}$. Since $j\in B_\rho$, by~\eqref{eq:avoid} this distance does not lie in $I_j=[2^{-j-1},2^{-j}]$, so it is less than $2^{-j-1}$. Hence $Y$ lies in an interval of length $2^{-j-1}$, and so $\operatorname{vol}(Y)\le\Phi(2^{-j-1})$. Taking the supremum over $\xi$ gives $\Phi(2^{-j})\le\Phi(2^{-j-1})$.

Finally, we bound $|B_\rho|$. Let $L\ge0$ be an integer and $\beta_L:=|\{0,\dots,L\}\cap B_\rho|$. Going from $\Phi(1)$ to $\Phi(2^{-L-1})$ in $L+1$ halvings, we use $\Phi(2\lambda)\le2\Phi(\lambda)$ for the $L+1-\beta_L$ indices $j\in\{0,\dots,L\}\setminus B_\rho$ and $\Phi(2^{-j})\le\Phi(2^{-j-1})$ for the $\beta_L$ indices $j\in \{0,\dots,L\}\cap B_\rho$. Together with $\Phi(2^{-L-1})\le2^{-L-1}$ this gives
\[
\sigma^n\le\operatorname{vol}(X)\le\Phi(1)\le 2^{L+1-\beta_L}\,\Phi(2^{-L-1})\le 2^{-\beta_L}.
\]
Therefore $\beta_L\le n\log_2(1/\sigma)=n\log_2(n/\rho)$ for all $L$. Since $|B_\rho|=\sup_L\beta_L$, this proves the theorem.
\end{proof}

For $j\ge0$ let $m_j:=\min_{\tau\in I_j}D_c(\tau)$, which exists since $D_c$ is continuous and $I_j$ is compact. The next corollary shows that the sum of these minima over all windows is $O(n^2)$, independently of $\norm c_\infty$.

\begin{corollary}\label{cor:sum}
For every $c\in\R^n$: $\displaystyle\sum_{j\ge0}m_j\ \le\ \Bigl(\frac12+\frac1{2\ln2}\Bigr)n^2\ <\ 1.222\,n^2$.
\end{corollary}
\begin{proof}
We have
\[
\begin{aligned}\sum_{j\ge0}m_j&=\sum_j\int_0^{\infty}\mathbf 1[m_j>\rho]\,d\rho=\int_0^{n/2}|B_\rho|\,d\rho\\&\le\int_0^{n/2}n\log_2\frac n\rho\,d\rho=n^2\!\int_0^{1/2}\!\log_2\frac1u\,du=n^2\Bigl(\frac12+\frac1{2\ln2}\Bigr).\end{aligned}
\]
The first equality holds since $a=\int_0^\infty\mathbf 1[a>\rho]\,d\rho$ for every $a\ge0$. The second follows by exchanging sum and integral, which is allowed since all terms are nonnegative (Tonelli's theorem), since $m_j>\rho$ if and only if $j\in B_\rho$ because the minimum is attained, and since $m_j\le n/2$ because $D_c$ takes values in $[0,n/2]$. The inequality is Theorem~\ref{thm:dirichlet}, applied for $0<\rho\le n/2$. The next equality follows from the substitution $\rho=nu$, and the last one from $\int_0^{1/2}\log_2\frac1u\,du=\frac1{\ln2}\bigl[u-u\ln u\bigr]_0^{1/2}$.
\end{proof}

\section{Proof of \texorpdfstring{Theorem~\ref{thm:main}}{Theorem 1}}\label{sec:proof}

\begin{proof}
We first reduce to the case that $P$ is a rational polytope. If $P=\emptyset$, then $\rk(P)=0$, so let $P\ne\emptyset$. Let $y\in\{0,1\}^n\setminus P$. Since $P$ is compact and convex, there is $a\in\R^n$ with $ay>\max_{x\in P}ax$. Both sides depend continuously on $a$, so we may take $a\in\mathbb{Q}^n$. Choosing a rational $\beta$ with $\max_{x\in P}ax<\beta<ay$, the rational half-space $\{x\in\R^n: ax\le\beta\}$ contains $P$ but not $y$. Let $Q$ be the intersection of $\cube$ with one such half-space for each $y\in\{0,1\}^n\setminus P$ (see also \cite[Section~2]{EisSch03}). Then $Q$ is a rational polytope with $P\subseteq Q\subseteq\cube$. Since $\cube\cap\Z^n=\{0,1\}^n$, we have $Q\cap\Z^n=P\cap\Z^n$, and hence $Q_I=\PI$. Since $P\subseteq Q$ implies $P'\subseteq Q'$, we have $P^{(k)}\subseteq Q^{(k)}$ for all $k$. Hence $\dep_P(c)\le\dep_Q(c)$ if $\PI\ne\emptyset$, and $\rk(P)\le\rk(Q)$ because $\PI\subseteq P^{(k)}\subseteq Q^{(k)}$ for all $k$. So we may assume that $P$ is a rational polytope.

Next we construct the chain to which we apply Lemma~\ref{lem:pipeline}. Let $\tau_0:=1$ and, for each integer $t\ge1$, let $\tau_t\in I_t$ be such that $D_c(\tau_t)=m_t$. For $t\ge0$ let $v_t\in\Z^n$ be such that $\norm{v_t-\tau_tc}_1=D_c(\tau_t)$, for instance a coordinatewise nearest integer vector to $\tau_tc$. Since $1\in I_0$ and $c\in\Z^n$, we have $D_c(\tau_0)=0=m_0$. Hence $v_0=c$ and $\norm{v_t-\tau_tc}_1=m_t$ for all $t\ge0$. Let $T$ be the least $t\ge1$ with $2^{-t}N<\frac12$, i.e.\ $T=\lfloor\log_2N\rfloor+2$. Then, since $\tau_T\le2^{-T}$ and $|c_\ell|\le N$, we have $|\tau_Tc_\ell|\le2^{-T}N<\frac12$ for all $\ell$, so $v_T=0$. Our chain consists of $v_0,\dots,v_T$ and the multipliers $\mu_t:=\tau_t/\tau_{t+1}$ for $0\le t<T$, as in Lemma~\ref{lem:pipeline}. Note that $0<\mu_t\le4$ for $0\le t<T$, since $\tau_t\le2^{-t}$ and $\tau_{t+1}\ge2^{-t-2}$.

Then we bound the errors $E_t$. Since $v_T=0$, we have $E_T=0$. For $t<T$, since $\mu_t\tau_{t+1}c=\tau_tc$ and $\mu_t\le4$,
\[
\norm{v_t-\mu_tv_{t+1}}_1=\norm{(v_t-\tau_tc)-\mu_t(v_{t+1}-\tau_{t+1}c)}_1\le m_t+4m_{t+1}.
\]
Hence $E_t\le1+m_t+4m_{t+1}$ for $t<T$, and by Lemma~\ref{lem:pipeline} and Corollary~\ref{cor:sum},
\[
\dep_P(c)\le 2n+2\sum_{t=0}^{T-1}\bigl(1+m_t+4m_{t+1}\bigr)\le 2n+2T+10\sum_{j\ge1}m_j\le 12.22\,n^2+2n+2\log_2N+4 .
\]

It remains to bound the rank. If $\PI=\emptyset$, then $\rk(P)\le n$ by Fact~\ref{f:empty}. Otherwise take the system $a_ix\le b_i$ of Fact~\ref{f:hadamard}. Since $\PI\subseteq P^{(k)}$ for all $k$, and $a_ix\le h(a_i)\le b_i$ is valid for $P^{(k)}$ whenever $k\ge\dep_P(a_i)$, we get $P^{(k)}=\PI$ for $k=\max_i\dep_P(a_i)$. Using $\log_2\norm{a_i}_\infty\le\frac n2\log_2n$ gives $\rk(P)\le 12.22\,n^2+n\log_2n+2n+4$.
\end{proof}


\begin{thebibliography}{99}
\bibitem{AloVu97} Noga Alon, Van H.~Vu, Anti-Hadamard matrices, coin weighing, threshold gates, and indecomposable hypergraphs, \emph{Journal of Combinatorial Theory, Series A} 79(1) (1997) 133--160, doi:10.1006/jcta.1997.2780.
\bibitem{BocEisHarSch99} Alexander Bockmayr, Friedrich Eisenbrand, Mark E.~Hartmann, Andreas S.~Schulz, On the Chv\'atal rank of polytopes in the 0/1 cube, \emph{Discrete Applied Mathematics} 98(1--2) (1999) 21--27, doi:10.1016/S0166-218X(99)00156-0.
\bibitem{Chv73} Va\v{s}ek Chv\'atal, Edmonds polytopes and a hierarchy of combinatorial problems, \emph{Discrete Mathematics} 4(4) (1973) 305--337, doi:10.1016/0012-365X(73)90167-2.
\bibitem{ChvCooHar89} Va\v{s}ek Chv\'atal, William J.~Cook, Mark E.~Hartmann, On cutting-plane proofs in combinatorial optimization, \emph{Linear Algebra and its Applications} 114/115 (1989) 455--499, doi:10.1016/0024-3795(89)90476-X.
\bibitem{ConCorZam14} Michele Conforti, G\'erard Cornu\'ejols, Giacomo Zambelli, \emph{Integer Programming}, Graduate Texts in Mathematics 271, Springer, Cham, 2014, doi:10.1007/978-3-319-11008-0.
\bibitem{CorLee16} G\'erard Cornu\'ejols, Dabeen Lee, On some polytopes contained in the 0,1 hypercube that have a small Chv\'atal rank, in: Quentin Louveaux, Martin Skutella (eds.), \emph{Integer Programming and Combinatorial Optimization: 18th International Conference, IPCO 2016}, Lecture Notes in Computer Science 9682, Springer, Cham, 2016, 300--311, doi:10.1007/978-3-319-33461-5\_25.
\bibitem{CorLee18} G\'erard Cornu\'ejols, Dabeen Lee, On some polytopes contained in the 0,1 hypercube that have a small Chv\'atal rank, \emph{Mathematical Programming, Series B} 172(1--2) (2018) 467--503, doi:10.1007/s10107-017-1226-4.
\bibitem{EisSch03} Friedrich Eisenbrand, Andreas S.~Schulz, Bounds on the Chv\'atal rank of polytopes in the 0/1-cube, \emph{Combinatorica} 23(2) (2003) 245--261, doi:10.1007/s00493-003-0020-5.
\bibitem{PadGro85} Manfred W.~Padberg, Martin Gr\"otschel, Polyhedral computations, in: Eugene L.~Lawler, Jan Karel Lenstra, Alexander H.~G.~Rinnooy Kan, David B.~Shmoys (eds.), \emph{The Traveling Salesman Problem: A Guided Tour of Combinatorial Optimization}, John Wiley \& Sons, Chichester, 1985, 307--360.
\bibitem{PokSch11} Sebastian Pokutta, Andreas S.~Schulz, Integer-empty polytopes in the 0/1-cube with maximal Gomory--Chv\'atal rank, \emph{Operations Research Letters} 39(6) (2011) 457--460, doi:10.1016/j.orl.2011.09.004.
\bibitem{PokSta11} Sebastian Pokutta, Gautier Stauffer, Lower bounds for the Chv\'atal--Gomory rank in the 0/1 cube, \emph{Operations Research Letters} 39(3) (2011) 200--203, doi:10.1016/j.orl.2011.03.001.
\bibitem{RotSan13} Thomas Rothvo\ss, Laura Sanit\`a, 0/1 polytopes with quadratic Chv\'atal rank, in: Michel X.~Goemans, Jos\'e Correa (eds.), \emph{Integer Programming and Combinatorial Optimization: 16th International Conference, IPCO 2013}, Lecture Notes in Computer Science 7801, Springer, Berlin, 2013, 349--361, doi:10.1007/978-3-642-36694-9\_30.
\bibitem{RotSan17} Thomas Rothvo\ss, Laura Sanit\`a, 0/1 polytopes with quadratic Chv\'atal rank, \emph{Operations Research} 65(1) (2017) 212--220, doi:10.1287/opre.2016.1549.
\bibitem{Sch80} Alexander Schrijver, On cutting planes, \emph{Annals of Discrete Mathematics} 9 (1980) 291--296, doi:10.1016/S0167-5060(08)70085-2.
\bibitem{Sch86} Alexander Schrijver, \emph{Theory of Linear and Integer Programming}, John Wiley \& Sons, Chichester, 1986.
\bibitem{Zie00} G\"unter M.~Ziegler, Lectures on 0/1-polytopes, in: Gil Kalai, G\"unter M.~Ziegler (eds.), \emph{Polytopes---Combinatorics and Computation}, DMV Seminar 29, Birkh\"auser, Basel, 2000, 1--41, doi:10.1007/978-3-0348-8438-9\_1.
\end{thebibliography}
\end{document}